\documentclass[12pt]{amsart}
\usepackage{hyperref, amsmath, amssymb, amsthm}

\newtheorem{proposition}{Proposition}
\newtheorem{theorem}{Theorem}

\DeclareMathOperator{\II}{II}       
\DeclareMathOperator{\Sgn}{sgn}     
\DeclareMathOperator{\St}{St}       
\DeclareMathOperator{\Sym}{Sym}     
\DeclareMathOperator{\Tr}{tr}       
\DeclareMathOperator{\Vol}{vol}     

\title{The Fourier Transform of the Stiefelian Surface Measure}
\author{Nikolaos Chatzikonstantinou}
\address{Okinawa Institute of Science and Technology Graduate University, Okinawa 904-0495, Japan}

\begin{document}

\maketitle

\begin{abstract}
  Let $\St^{n}_{k}\subset\mathbb{R}^{n\times k}$ be the set of all $n\times k$ matrices whose columns are mutually orthogonal and of unit Euclidean length, and let $\mu_{n,k}$ be the surface measure corresponding to this embedding. We calculate the first term of the asymptotic expansion of the Fourier transform of $\mu_{n,k}$ for most directions, using the method of stationary phase. The asymptotic behavior near the remaining directions is unknown. We note some interesting connections to trace moments of orthogonal matrices, discrete random walks, Bessel functions, and pose some questions.
\end{abstract}

\tableofcontents

\section{Introduction}
\label{sec:introduction}
  Let $S^{n-1} := \{(x_{1},\dots, x_{n}) \in \mathbb R^{n} : x^{2}_{1} + \cdots + x^{2}_{n} = 1\}$ denote the unit sphere in $\mathbb R^{n}$, $n\geq 2$. This submanifold inherits a surface measure, denoted by $\sigma_{n-1}$. We define the \emph{Fourier transform} of a measure $\mu$ on $\mathbb R^{n}$ by the formula
  \begin{align}
    \label{eq:fourier-transform}
    \widehat\mu(\xi) & := \int e^{-2\pi i x\cdot\xi}d\mu(x), & \xi\in\mathbb R^{n}.
  \end{align}
  A simple computation gives
  \begin{align}
    \label{eq:sphere-bessel}
    \widehat{\sigma}_{n-1}(\xi) = 2\pi|\xi|^{-\frac{n-2}2}J_{\frac{n-2}2}(2\pi|\xi|),
  \end{align}
  where $J_{\nu}(t)$ is the \emph{Bessel function of first kind} and $\nu$-th order, see~\cite[Appendix B]{cit:grafakos-cfa}. From~\eqref{eq:sphere-bessel} one obtains the asymptotic expansion
  \begin{align}
    \label{eq:asymptotic-expansion}
    \widehat{\sigma}_{n-1}(\xi) & = 2\cos(2\pi(|\xi| - \frac{n-1}8))|\xi|^{-\frac{n-1}2} + O(|\xi|^{-\frac{n+1}2}), & \xi\to\infty.
  \end{align}
  Note that~\eqref{eq:asymptotic-expansion} does not make any claims about values at specific points. A full asymptotic formula is available for $J_{\nu}(t)$ (see~\cite[Chapter VIII]{cit:stein}) and consequentially for~\eqref{eq:sphere-bessel} as well, and the derivates may similarly be estimated using the identity $\frac d{dt} J_{\nu}(t) = \frac{\nu}2 J_{\nu}(t) - J_{\nu+1}(t)$.

  The above formulas may be useful in any operator or integral that involves the unit sphere. For some applications, see~\cite{cit:mattila}. It is of interest to replace the unit sphere with other submanifolds satisfying a certain property, and the \emph{method of stationary phase} allows us to again claim full asymptotic expansion formulas, albeit whose terms may be difficult to compute explicitly past the first one. We are inspired by problems and results involving configurations (see for instance~\cite{cit:iosevich-liu},~\cite{cit:palsson-sovine}), to seek to estimate for $n\geq 3$ and $2\leq k \leq n$ the following integral,

  \begin{align}
    \label{eq:stiefel-fourier}
    \widehat{\mu}_{n,k}(\Xi) & = \int_{\St^{n}_{k}} e^{-2\pi i\Tr (X^{t}\Xi)} d\mu_{n,k}, & \Xi \in \mathbb{R}^{n\times k},
  \end{align}
  where $\Tr$ denotes trace, $\St^{n}_{k} := \{X\in\mathbb{R}^{n\times k} : X^{t}X = I_{k}\}$ and $I_{k}$ is the $k\times k$ identity matrix. This set is a generalization of the unit sphere above, as $\St^{n}_{1} = S^{n-1}$. The set $\St^{n}_{k}$ is called the \emph{Stiefel manifold} and is the moduli space of unit $k$-frames in $\mathbb{R}^{n}$. By applying the method of stationary phase, we obtain as $\|\Xi\|\to\infty$ the formula,
  \begin{align}
    \label{eq:stiefel-asymptotics}
    \widehat{\mu}_{n,k}(\Xi) =&~2^{\frac{k(k-1)}4}\sum_{s\in\{-1, 1\}^{k}}\cos(2\pi\sum_{j=1}^{k}s_{j}(\lambda_{j} - \frac{n-j}8))~\times \nonumber \\
    & |\lambda_{1}\cdots\lambda_{k}|^{-\frac{n-k}2}\prod_{1\leq i < j\leq k}|s_{i}\lambda_{i} - s_{j}\lambda_{j}|^{-1/2} + O(\|\Xi\|^{-\frac{n-k+2}2}),
  \end{align}
  where $\lambda_{1},\dots,\lambda_{k}$ are the singular values of $\Xi$. The decay given by~\eqref{eq:stiefel-asymptotics} has the defect that it is not useful in two cases: when $\lambda_{i} \approx \lambda_{j}$ for some $i,j$ or when for some $i$ we have $\lambda_{i} \approx 0$. In those cases we are not able to obtain any results, although we expect that the result will be similar to~\eqref{eq:stiefel-asymptotics} with the division-by-zero terms deleted from the formula. Our expectation is based on two facts, a precise computation for $n=4,k=2$ (see~\eqref{eq:explicit-four}) and that when $\lambda_{k}$ is exactly equal to zero, we have a reduction
  \begin{align}
    \label{eq:reduction}
    \widehat{\mu}_{n,k}(\Xi) = \Vol(S^{n-k})~\widehat{\mu}_{n,k-1}(
    \begin{pmatrix}
      \lambda_{1}&&\\&\ddots&\\&&\lambda_{k-1}\\&\vdots&
    \end{pmatrix}),
  \end{align}
  where the other entries of the $n\times(k-1)$ matrix are zeros. This reduction may be iterated, and so it is reasonable to assume $k$ positive singular values. Also of interest is the ``most generate'' case $\lambda_{1} = \cdots = \lambda_{k}$, whose asymptotic expansion will say something about the growth of the moments of the trace function in $\mathbb{O}(k)$ (see for instance \cite{cit:pastur-vasilchuk}), since
  \begin{align}
    \label{eq:moments}
    \frac 1{\Vol{\mathbb{O}(k)}}\int_{\mathbb{O}(k)} e^{i\lambda\Tr X}dX & = \sum_{m=0}^{\infty} \frac{(i\lambda)^{m}}{m!}\mathbb E[(\Tr X)^{m}].
  \end{align}
  As a method of attacking the above we have some details presented in Section~\ref{sec:closer-look-at} on making explicit calculations, as the authors of~\cite{cit:palsson-sovine} did for $k=2$.

  \section{Proof of Main Result}
  \label{sec:proof}

    \subsection{Overview}
    \label{sec:overview}

    We wish to prove~\eqref{eq:stiefel-asymptotics}.

    Symmetry reduces the form of $\Xi$ to a rectangular-diagonal form with nonincreasing entries $\lambda_{1}\geq \cdots \geq \lambda_{k}\geq 0$. Fubini's theorem allows us to integrate out the columns of $X$ for which $\lambda_{j} = 0$, and so we may assume $\lambda_{k} > 0$. We calculate the tangent and normal planes of $\St^{n}_{k}$ and obtain an expression in matrix calculus for their projectors and for the second fundamental form. We calculate the points on $\St^{n}_{k}$ for which $\Xi$ belongs to the normal plane, and there we calculate the eigenvalues of the second fundamental form. The calculation is simple because it turns out to be diagonal. Finally using these eigenvalues we apply Theorem~\ref{thm:stationary-phase} to obtain~\eqref{eq:stiefel-asymptotics}. A reference for the geometric concepts of this section is~\cite{cit:docarmo}.

    \subsection{Notation}
    \label{sec:notation}

    For a matrix $X\in\mathbb{R}^{n\times k}$, denote by $X_{1}\in\mathbb{R}^{k\times k}$ and $X_{2}\in\mathbb{R}^{(n-k)\times k}$ the matrices such that $X = \begin{pmatrix}X_{1}\\X_{2}\end{pmatrix}$. Let $Y\in\St^{n}_{k}$ denote the rectangular-diagonal $n\times k$ matrix with entries equal to $1$. Let $\mathfrak{so}(k)$ denote the skew-symmetric $k\times k$ matrices with real entries, and $\Sym^{2}(k)$ denote the symmetric $k\times k$ matrices. Let $\nabla$ be the standard connection of Euclidean space. Somewhat redundantly, the inner product is available to us both in the matrix calculus and by the $\langle\cdot,\cdot\rangle$ symbol. The Kronecker delta is denoted by $\delta_{ij}$.

    \subsection{Reduction to Singular Values}
    \label{sec:reduct-diag-matr}

    The map $X\mapsto OXP$ for any $O\in\mathbb{O}(n)$, $P\in\mathbb{O}(k)$ preserves the measure $\mu_{n,k}$ and so by a change of variables we obtain
    \begin{align}
      \label{eq:symmetries}
      \widehat{\mu}_{n,k}(\Xi) & = \widehat{\mu}_{n,k}(O\Xi P).
    \end{align}
    Using the singular value decomposition of $\Xi$, we may choose $O,P$ so that $O\Xi P$ is rectangular-diagonal with nonincreasing entries $\lambda_{1}\geq\cdots\geq\lambda_{k}\geq 0$. Let $\Pi(j)$, $1 \leq j \leq k$ be the plane that orthogonally complements the plane spanned by the first $j$ columns of $X\in\St^{n}_{k}$, and $\sigma_{\Pi(j)}$ be the measure of the unit sphere of $\Pi(j)$. We have the equality
    \begin{align}
      \label{eq:measure-decomposition}
      \mu_{n,k} = d\sigma_{\Pi(k-1)}(x_{k})\cdots d\sigma_{\Pi(1)}(x_{2})d\sigma_{\mathbb{R}^{n}}(x_{1})
    \end{align}
    where $x_{1},\dots,x_{k}$ are the columns of $X$, and so if we have $\lambda_{k} = \lambda_{k-1} = \cdots = \lambda_{k_{0}+1} = 0$ for some $k_{0}$, then by integrating those columns out, we obtain
    \begin{align}
      \label{eq:reduction2}
      \widehat{\mu}_{n,k}(\Xi) & = \left(\prod_{j=k_{0}+1}^{k} \Vol(S^{n-j})\right)~\widehat{\mu}_{n,k_{0}}(
                                 \begin{pmatrix}
                                   \lambda_{1}&&\\&\ddots&\\&&\lambda_{k_{0}}\\&\vdots&
                                 \end{pmatrix}
                                 ),
    \end{align}
    where the rest of the elements of the matrix in~\eqref{eq:reduction2} are zeros. Thus we may assume that $\Xi$ takes the special form
    \begin{align}
      \label{eq:singular-values}
      \Xi & =
            \begin{pmatrix}
              \lambda_{1}&&\\&\ddots&\\&&\lambda_{k}\\&\vdots&
            \end{pmatrix}, & \lambda_{1}\geq\cdots\geq\lambda_{k}>0.
    \end{align}

    \subsection{Second Fundamental Form}
    \label{sec:second-order-study}

    The trick to computing first and second-order expansions is to express the Gauss map and second fundamental form using matrix calculus, and this approach is from~\cite{cit:absil-mahony-sepulchre}. It may help to put this in the context of the sphere. For a vector $u\in\mathbb{R}^{n}$ the spherical tangential projector at $x\in S^{n-1}$ is given by $P_{x}(u) = (I_{n} - xx^{t})u$, and by differentiating in the direction of $v\in T_{x}S^{n-1}$ we obtain, for unit basis vectors $e_{i}, e_{j}$ of the tangent space at $x$,
    \begin{align*}
      \langle \nabla_{e_{i}}P_{x}(e_{j}), x\rangle
      & = -\langle e_{i}x^{t}e_{j} + xe_{i}^{t}e_{j}, x\rangle \\
      & = -\delta_{ij},
    \end{align*}
    which is exactly $\langle\II_{x}(e_{i}, e_{j}),\xi\rangle$ where $\xi~(=x)$ is the outward unit normal at $x$. Indeed, one may write the sphere near $(0,\dots,0,1)$ as the graph of a function
    \begin{align*}
      x & \mapsto 1-\delta_{ij}x_{i}x_{j}/2 + O(|x|^{3}), & x\in\mathbb{R}^{n-1}.
    \end{align*}

      \subsubsection{Tangent and Normal Planes}
      \label{sec:tang-norm-plan}

      By differentiating the equation $X^{t}X = I_{k}$ we obtain $dX^{t}X + X^{t}dX = 0$, and taking $X = Y$ gives $dX^{t}_{1} + dX_{1} = 0$. Thus the tangent plane of $\St^{n}_{k}$ at $Y$ is given by all matrices of the form $\begin{pmatrix}A_{1}\\A_{2}\end{pmatrix}$, where $A_{1}\in\mathfrak{so}(k)$ and $A_{2}\in\mathbb{R}^{(n-k)\times k}$. Since symmetric matrices orthogonally complement skew-symmetric matrices, the normal space is given by all matrices of the form $\begin{pmatrix}N_{1}\\0\end{pmatrix}$ where $N_{1}\in\Sym^{2}(k)$. Using the orthogonal decomposition $\mathbb{R}^{k\times k}\ni L \mapsto (\frac{L-L^{t}}2, \frac{L+L^{t}}2)$ and the equations $Y^{t}A = A_{1}$ and $YA_{1} = \begin{pmatrix}A_{1}\\0\end{pmatrix}$, one may write the projector at $Y$ by
      \begin{align}
        \label{eq:tangential-projector}
        P_{Y}(A) = (I_{n} - YY^{t})A+\frac 12Y(Y^{t}A - A^{t}Y).
      \end{align}
      The normal projector is given by $P^{\bot} = I_{n} - P$, and so
      \begin{align}
        \label{eq:normal-projector}
        P^{\bot}_{Y}(A) = \frac 12 Y(Y^{t}A + A^{t}Y).
      \end{align}
      In fact, formulas~\eqref{eq:tangential-projector} and~\eqref{eq:normal-projector} hold even when the special point $Y$ is replaced by a general point $X\in\St^{n}_{k}$, which one sees by writing $Y = OX$ for some $O\in\mathbb{O}(n)$ and then for $B\in T_{X}\St^{n}_{k}$ setting $A = O^{t}B$ in~\eqref{eq:tangential-projector},~\eqref{eq:normal-projector}. It follows that the form of the tangent and normal spaces at a general point $X\in\St^{n}_{k}$ is given by
      \begin{align}
        \label{eq:tangent-space}
        T_{X}\St^{n}_{k} & = \{XA_1 + KA_2 : A_1\in\mathfrak{so}(k), A_2\in\mathbb{R}^{(n-k)\times k}\}, & X^{t}K = 0, \\
        \label{eq:normal-space}
        T^{\bot}_{X}\St^{n}_{k} & = \{XN_1 : N_1\in\Sym^2(k)\}.
      \end{align}
      where $K$ is any $n\times (n-k)$ matrix of maximal rank solving $X^{t}K = 0$.

      \subsubsection{Differentiating the Projector}
      \label{sec:diff-proj}

      Let $X\in\St^{n}_{k}$, $A,B\in T_{X}\St^{n}_{k}$ and extend $B$ to the vector field $\bar B := P_{Z}B, Z\in\mathbb{R}^{n\times k}$. We define the second fundamental form of $\St^{n}_{k}$ at $X$ in the usual manner,
      \begin{align}
        \label{eq:second-fundamental-form}
        \II_{X}(A,B) & := P^{\bot}_{X}(\nabla_{A}\bar{B})
      \end{align}
      and claim that
      \begin{align}
        \label{eq:differentiating-projector}
        \II_{X}(A,B) = \nabla_{A}P_{X}(B),
      \end{align}
      which one can see by noting that the map $\St^{n}_{k}\ni Z\mapsto \|P_{Z}(B)\|$ has a critical point at $Z = X$, where $P_{X}(B) = B$. This allows us to compute,
      \begin{proposition}[from~\cite{cit:absil-mahony-sepulchre}]
        \label{prop:second-fundamental-form}
        For $X\in\St^{n}_{k}$ and $A,B\in T_{X}\St^{n}_{k}$,
        \begin{align}
          \label{eq:second-fundamental-form-matrix}
          \II_{X}(A,B) = - \frac 12 X(A^{t}B + B^{t}A).
        \end{align}
      \end{proposition}
      \begin{proof}
        It suffices to prove the formula at $X = Y$, as it follows for other points by the homogeneity of $\St^{n}_{k}$, as in Section~\ref{sec:tang-norm-plan}. By~\eqref{eq:differentiating-projector} and~\eqref{eq:tangential-projector},
        \begin{align}
          \label{eq:second-fundamental-form-calculation}
          \II_{Y}(A,B)
          & = -AY^{t}B - YA^{t}B + \frac 12 A(Y^{t}B - B^{t}Y) + \frac 12 Y(A^{t}B - B^{t}A) \\
          & = -\frac 12(A(Y^{t}B + B^{t}Y) + Y(A^{t}B + B^{t}A)).
        \end{align}
        The first term is zero since $Y^{t}B = B_{1}\in\mathfrak{so}(k)$ according to Section~\ref{sec:tang-norm-plan}.
      \end{proof}

      \subsubsection{Calculating the Eigenvalues}
      \label{sec:calc-eigenv}

      The points $X\in\St^{n}_{k}$ for which $\Xi$ belongs to the normal space are given by the rectagular-diagonal matrices with diagonal elements $\pm 1$, since according to~\eqref{eq:normal-space} they must solve the equation $XN_{1} = \Xi$ for some $N_{1}\in\Sym^{2}(k)$, which means that $N_{1}^{2} = \Xi^{t}\Xi$ and so the fact follows by the diagonal reduction~\eqref{eq:singular-values}. In particular $Y$ is one of them. In fact, they may all be parametrized by $s\in\{-1,1\}^{k}$, as they are all the rectangular-diagonal matrices with $\pm 1$ as their diagonal elements. We calculate $\langle\II_{Y}, \Xi\rangle$, as the other cases follow by the homogeneity property of $\St^{n}_{k}$ as in Section~\ref{sec:tang-norm-plan}.

      We wish to calculate $\langle\II_{Y}(A,B), \Xi\rangle$ for $A,B$ elements of a basis of the tangent space at $Y$. Luckily, only the diagonal terms are nonzero for the obvious choice of basis: for the $\mathfrak{so}(k)$-part of the tangent space, the basis is given by $\mathcal A_{i,j} := (E_{i,j} - E_{j,i})/\sqrt{2}$ for $1 \leq i < j \leq k$, where $E_{i,j}\in\mathbb{R}^{n\times k}$ is the matrix with $1$ in the $i$-th row and $j$-th column and zeros elsewhere, and for the Grassmannian part, the rest of the basis is given by $E_{i,j}$ for $k < i \leq n$ and $1 \leq j \leq k$. It is then easy to verify using~\eqref{eq:second-fundamental-form-matrix} the basis relations,
      \begin{align}
        \label{eq:eigenvalues-so-part}
        \langle \II_{Y}(\mathcal{A}_{i,j}, \mathcal{A}_{i,j}), \Xi\rangle & = - \frac 12(\lambda_{i} + \lambda_{j}), & 1 \leq i < j \leq k, \\
        \label{eq:eigenvalues-grassman-part}
        \langle \II_{Y}(E_{i,j}, E_{i,j}), \Xi\rangle & = - \lambda_{j}, & 1 \leq j \leq k < i \leq n, \\
        \langle \II_{Y}(A, B), \Xi\rangle & = 0, &\text{other basis vectors}. \nonumber
      \end{align}
      For the other points parametrized by $s\in\{-1,1\}^{k}$, replace $\lambda_{i}$ with $s_{i}\lambda_{i}$ in~\eqref{eq:eigenvalues-so-part} and~\eqref{eq:eigenvalues-grassman-part}. Indeed the signature of the second fundamental form for such a point is easily computed to be equal to
      \begin{align}
        \label{eq:signature}
        \sum_{j=1}^{k}s_{j}(j-n),
      \end{align}
      and the absolute value of the determinant of its second fundamental form to be equal to
      \begin{align}
        \label{eq:determinant}
        2^{-k(k-1)/2}|\lambda_{1}\cdots\lambda_{k}|^{n-k}\prod_{1\leq i < j\leq k}|s_{i}\lambda_{i} + s_{j}\lambda_{j}|.
      \end{align}

    \subsection{Application of the Method of Stationary Phase}
    \label{sec:appl-meth-stat}

    Theorem~\ref{thm:stationary-phase}, stated below, may be found for $m=n-1$ in~\cite[Chapter VIII]{cit:stein}. It follows for general $m$ by simple geometric considerations, although we could not find a reference that contains it in this exact form. We apply Theorem~\ref{thm:stationary-phase} using the calculations of Section~\ref{sec:calc-eigenv}, in particular~\eqref{eq:signature} and~\eqref{eq:determinant}, to obtain~\eqref{eq:stiefel-asymptotics}.
    \begin{theorem}[Stationary Phase]
      \label{thm:stationary-phase}
      Let $M \hookrightarrow \mathbb{R}^n$ be an immersed $m$-dimensional compact submanifold of $\mathbb{R}^n$ with surface measure denoted by $\mu$. Let $\xi\in\mathbb{R}^n$ be a unit vector and let $\langle\II_x, \xi\rangle$ denote the second fundamental form of $M$ at $x\in M$ in the direction of $\xi$. Assume there are finitely many points $x_1, \dots, x_N$ where $\xi$ belongs to the normal space of $M$ and assume that $\det\langle \II, \xi\rangle \not = 0$ there. Denote by $\Sgn\langle \II_{x}, \xi\rangle$ the number of positive eigenvalues minus the number of negative eigenvalues of the second fundamental form at $x$. Then for $\tau\to+\infty$ and continuously in $\xi$,
      \begin{align*}
        \widehat{\mu}(\tau\xi) & = \tau^{-m/2}~\sum_{j=1}^N e^{-2\pi i(\tau x_{j}\cdot \xi + \Sgn\langle\II_{x_j}, \xi\rangle/8)}|\det\langle\II_{x_j},\xi\rangle|^{-1/2} + O(\tau^{-m/2-1}).
      \end{align*}
    \end{theorem}

  \section{A Closer Look at the Problematic Singular Values}
  \label{sec:closer-look-at}

    The only explicit computation we are able to carry out presently is for $k=2$, and this was first done in~\cite{cit:palsson-sovine}. We find using~\eqref{eq:measure-decomposition} that
    \begin{align}
      \label{eq:explicit-computation}
      \widehat{\mu}_{n,2}(\begin{pmatrix}\kappa & \\ & \lambda \\ \vdots & \vdots\end{pmatrix}) & =
      \int e^{-2\pi i (\kappa x\cdot e_{1} + \lambda y\cdot e_{2})} d\mu_{n,2}(x,y) \\
      & = \iint e^{-2\pi i\lambda y\cdot e_{2}} d\sigma_{\bot,x}(y) e^{2\pi i \kappa x\cdot e_{1}} d\sigma_{n-1}(x)
    \end{align}
    where $\sigma_{\bot,x}$ is the surface measure of the unit $(n-2)$-sphere of the hyperplane perpendicular to $x$. It is easy to see that this is equal to
    \begin{align}
      \label{eq:explicit-computation-2}
      \int e^{-2\pi i \kappa x\cdot e_{1}} \widehat{\sigma}_{n-2}(\lambda\sqrt{1-x_{2}^{2}})d\sigma_{n-1}(x).
    \end{align}
    By applying the coarea formula on the function $f(x) = x_{2}$, we further obtain
    \begin{align}
      \label{eq:explicit-computation-3}
      \Vol(S^{n-2})~\int_{-1}^{1} \widehat{\sigma}_{n-2}(\kappa\sqrt{1-t^{2}})\widehat{\sigma}_{n-2}(\lambda\sqrt{1-t^{2}})(1-t^{2})^{\frac{n-3}2}dt,
    \end{align}
    which according to~\eqref{eq:sphere-bessel} is also equal to
    \begin{align}
      \label{eq:explicit-computation-4}
      \frac{4\pi^{2}\Vol(S^{n-2})}{(\kappa\lambda)^{\frac{n-3}2}}~\int_{-1}^{1} J_{\frac{n-3}2}(2\pi\kappa\sqrt{1-t^{2}})J_{\frac{n-3}2}(2\pi\lambda\sqrt{1-t^{2}})dt.
    \end{align}
    The authors in~\cite{cit:palsson-sovine} note that they can not extract the full asymptotic decay from~\eqref{eq:explicit-computation-4}. On the other hand, the authors in~\cite{cit:iosevich-liu} obtain our result~\eqref{eq:stiefel-asymptotics} for $k=2$ using the method of stationary phase.

  \section{Future Directions}
  \label{sec:future-directions}
  For Section~\ref{sec:closer-look-at} we pose three questions:
    \begin{enumerate}
      \item Can the asymptotic decay predicted by~\eqref{eq:stiefel-asymptotics} be recovered from~\eqref{eq:explicit-computation-4}? Indeed, one could hope for the full asymptotic expansion, not just the first term. Presumably progress here would require identities involving Bessel functions, see below for one such case.
      \item Can one claim some (reduced) decay for $\lambda\approx 0$ or for $\kappa\approx\lambda$ using~\eqref{eq:explicit-computation-4}? Even the case $\kappa=\lambda$ would be interesting.
      \item Repeat the above two questions for $k\geq 3$: can $\widehat{\mu}_{n,k}$ be written out explicitly in terms of Bessel functions and can its full asymptotic expansion be computed? Can the cases $\lambda_{i}\approx 0$ and $\lambda_{i}\approx\lambda_{j}$ (for some $i\not= j$) be dealt with? The case $k=3$ seems to be within reach.
    \end{enumerate}
    We have the following remarks to make. First, using $J_{1/2}(t) = \sqrt{\frac 2{\pi t}}\sin(t)$, one obtains for $n=4$ that
    \begin{align}
      \label{eq:explicit-four}
      \widehat{\mu}_{4,2}(\Xi)
      & = \frac{16\pi}{\kappa\lambda}[J_{0}(2\pi(\kappa-\lambda))-J_{0}(2\pi(\kappa+\lambda))],
    \end{align}
    from which we may obtain the full asymptotic expansion. Second, by setting $\tilde{\sigma}_{r}$ to be the normalized probability measure on the sphere of radius $r>0$ centered at the origin, we may rewrite~\eqref{eq:explicit-computation-3} into a convolution,
    \begin{align}
      \label{eq:random-walk}
      \Vol(S^{n-2})\Vol(S^{n-1})^{2}\int_{-1}^{1} (\tilde{\sigma}_{\kappa}*\tilde{\sigma}_{\lambda})^{\vee}(\sqrt{1-t^{2}})(1-t^{2})^{\frac{n-3}2}dt,
    \end{align}
    which connects the computation with a 2-step random walk, which is an interesting problem in itself, see~\cite{cit:kluyver},~\cite{cit:borwein-et-al}. From this observation we may anticipate the singular behavior at $\kappa = \lambda$ for all $n\geq 3$ and $k=2$.

    If one can indeed achieve estimates for $\lambda_{i}\approx 0$ and $\lambda_{i}\approx\lambda_{j}$, then those may be applied (as the authors in~\cite{cit:palsson-sovine} do for $k=2$) using~\cite[Theorem~1.1]{cit:grafakos-he-honzik-park} to obtain continuity results for simplicial operators for $k\geq 3$ analogous to~\cite{cit:palsson-sovine} for $k=2$. Such results so obtained may also be contrasted with those obtained by the authors in~\cite{cit:iosevich-palsson-sovine}.

    Finally, one may compare this result to a related formula known as the Harish-Chandra-Itzykson-Zuber integral formula, see~\cite{cit:harish-chandra} and~\cite{cit:itzykson-zuber}, which is an exact formula for the integral $\int_{\mathbb{U}(n)}e^{\lambda \Tr(AUBU^*)}dU$. This formula has a quadratic phase and is over the unitary group instead of the real orthogonal group, in contrast to our result. One could seek results of our kind for $\mathbb{O}(n)$ replaced by $\mathbb{U}(n)$ and vice versa using $\mathbb{O}(n)$ for the quadratic phase.

    \section*{Acknowledgements}
    The author wishes to thank Allan Greenleaf for helpful discussions on the method of stationary phase and the second-order geometry of the Stiefel manifold and Daniel Spector for pointing out~\eqref{eq:explicit-four} and their encouragement to work on this problem.
    \label{sec:acknowledgements}


\begin{thebibliography}{99}
  \bibitem{cit:absil-mahony-sepulchre} P.-A. Absil, R. Mahony\ and\ R. Sepulchre, {\it Optimization algorithms on matrix manifolds}, Princeton University Press, Princeton, NJ, 2008. MR2364186
  \bibitem{cit:pastur-vasilchuk} L. Pastur\ and\ V. Vasilchuk, On the moments of traces of matrices of classical groups, Comm. Math. Phys. {\bf 252} (2004), no.~1-3, 149--166. MR2104877
  \bibitem{cit:iosevich-liu} A. Iosevich\ and\ B. Liu, Equilateral triangles in subsets of $\Bbb R^d$ of large Hausdorff dimension, Israel J. Math. {\bf 231} (2019), no.~1, 123--137. MR3960003
  \bibitem{cit:stein} E. M. Stein, {\it Harmonic analysis: real-variable methods, orthogonality, and oscillatory integrals}, Princeton Mathematical Series, 43, Princeton University Press, Princeton, NJ, 1993. MR1232192
  \bibitem{cit:docarmo} M. P. do Carmo, {\it Differential geometry of curves \& surfaces}, revised \& updated second edition, Dover Publications, Inc., Mineola, NY, 2016. MR3837152
  \bibitem{cit:palsson-sovine} E. A. Palsson\ and\ S. R. Sovine, The triangle averaging operator, J. Funct. Anal. {\bf 279} (2020), no.~8, 108671, 21 pp. MR4109092
  \bibitem{cit:grafakos-cfa}  L. Grafakos, {\it Classical Fourier analysis}, third edition, Graduate Texts in Mathematics, 249, Springer, New York, 2014. MR3243734
  \bibitem{cit:mattila} P. Mattila, {\it Fourier analysis and Hausdorff dimension}, Cambridge Studies in Advanced Mathematics, 150, Cambridge University Press, Cambridge, 2015. MR3617376
  \bibitem{cit:kluyver} J. C. Kluyver, {\it A local probability problem}, Ned. Acad. Wet. Proc. {\bf 8}, 341-350 (1906)
  \bibitem{cit:borwein-et-al} J. M. Borwein\ et al., Some arithmetic properties of short random walk integrals, Ramanujan J. {\bf 26} (2011), no.~1, 109--132. MR2837721
  \bibitem{cit:harish-chandra} Harish-Chandra, Differential operators on a semisimple Lie algebra, Amer. J. Math. {\bf 79} (1957), 87--120. MR0084104
  \bibitem{cit:itzykson-zuber} C. Itzykson\ and\ J. B. Zuber, The planar approximation. II, J. Math. Phys. {\bf 21} (1980), no.~3, 411--421. MR0562985
  \bibitem{cit:grafakos-he-honzik-park} L. Grafakos, D. He, P. Honz\'{i}k and B. J. Park, {\it Initial $L^{2}\times\cdots\times L^{2}$ bounds for multilinear operators}, (2020), arXiv:2010.15312 {\bf [math.CA]}
  \bibitem{cit:iosevich-palsson-sovine} A. Iosevich, E. A. Palsson and S. R. Sovine, {\it Simplex Averaging Operators: Quasi-Banach and $L^{p}$-Improving Bounds in Lower Dimensions}, (2021), arXiv:2109.09017 {\bf [math.CA]}
\end{thebibliography}
\end{document}